\documentclass[final,notitlepage,12pt,reqno,tbtags]{amsart}
\usepackage{graphicx}
\usepackage{calc}
\usepackage{url}
\usepackage{mathrsfs}
\usepackage{rotating}
\newlength{\depthofsumsign}
\makeatletter
\let\I\@undefined
\makeatother
\usepackage{geometry}
\usepackage{indentfirst}
\usepackage[normalem]{ulem}
\usepackage{float}
\usepackage{amsthm}

\usepackage{enumitem}[2011/09/28]
\setenumerate{align=left, leftmargin=0pt,labelsep=.5em, labelindent=0\parindent,listparindent=\parindent,itemindent=*}
\usepackage{array}
\usepackage{empheq}
\usepackage{natbib}
\usepackage[all]{xy}

\usepackage{multirow}

\newbox\shell
\newcommand{\dia}[2]{\setbox\shell=\hbox{\begin{picture}(180,120)(-90,-60)#1
\put(-90,-60){\makebox(180,120)[b]{\large #2}}\end{picture}}\dimen0=\ht
\shell\multiply\dimen0by7\divide\dimen0by16\raise-\dimen0\box\shell\hfill}

\usepackage{caption}[2012/02/19]

\usepackage{longtable,lscape}

\usepackage{mathptmx}       
\DeclareSymbolFont{operators}{OT1}{txr}{m}{n}
\SetSymbolFont{operators}{bold}{OT1}{txr}{bx}{n}
\def\operator@font{\mathgroup\symoperators}
\DeclareSymbolFont{italic}{OT1}{txr}{m}{it}
\SetSymbolFont{italic}{bold}{OT1}{txr}{bx}{it}
\DeclareSymbolFontAlphabet{\mathrm}{operators}
\DeclareMathAlphabet{\mathbf}{OT1}{txr}{bx}{n}
\DeclareMathAlphabet{\mathit}{OT1}{txr}{m}{it}
\SetMathAlphabet{\mathit}{bold}{OT1}{txr}{bx}{it}
\DeclareSymbolFont{letters}{OML}{txmi}{m}{it}
\SetSymbolFont{letters}{bold}{OML}{txmi}{bx}{it}
\DeclareFontSubstitution{OML}{txmi}{m}{it}
\DeclareSymbolFont{lettersA}{U}{txmia}{m}{it}
\SetSymbolFont{lettersA}{bold}{U}{txmia}{bx}{it}
\DeclareFontSubstitution{U}{txmia}{m}{it}
\DeclareSymbolFontAlphabet{\mathfrak}{lettersA}
\DeclareSymbolFont{symbols}{OMS}{txsy}{m}{n}
\SetSymbolFont{symbols}{bold}{OMS}{txsy}{bx}{n}
\DeclareFontSubstitution{OMS}{txsy}{m}{n}

\usepackage{amssymb,bm,amsmath}
\usepackage{amsfonts}

\usepackage[OT2,T1,T2A]{fontenc}
\usepackage[utf8x]{inputenc}

\usepackage[english,french,german,russian]{babel}
\usepackage{appendix}

\DeclareMathOperator{\Tr}{Tr}

\DeclareMathOperator{\D}{d}
\DeclareMathOperator{\I}{Im}

\def\XXint#1#2#3{{\setbox0=\hbox{$#1{#2#3}{\int}$}
     \vcenter{\hbox{$#2#3$}}\kern-.5\wd0}}

\def\crb{b\hspace{-0.5em}{^\text{--}}}

\bibpunct{[}{]}{,}{n}{}{;}

\def\eor{\hfill$ \square$}

\theoremstyle{plain}
\newtheorem{theorem}{Theorem}[section]
\newtheorem{proposition}[theorem]{Proposition}

\newtheorem{corollary}[theorem]{Corollary}

\newenvironment{remark}[1][Remark]{\begin{trivlist}
\item[\hskip \labelsep {\bfseries #1}]}{\end{trivlist}}

\theoremstyle{definition}

\numberwithin{equation}{section}

\usepackage{tikz}

\tikzset{>=stealth}
\usepackage{pgfplots}

\usepackage{color}

\begin{document}

\selectlanguage{english}
\title{Two inequalities for convex equipotential surfaces}
\author{Yajun Zhou}
\address{Program in Applied and Computational Mathematics (PACM), Princeton University, Princeton, NJ 08544; Academy of Advanced Interdisciplinary Studies (AAIS), Peking University, Beijing 100871, P. R. China }
\email{yajunz@math.princeton.edu, yajun.zhou.1982@pku.edu.cn}

\thanks{
\\\indent * This research was supported in part  by the Applied Mathematics Program within the Department of Energy
(DOE) Office of Advanced Scientific Computing Research (ASCR) as part of the Collaboratory on
Mathematics for Mesoscopic Modeling of Materials (CM4)}
\date{\today}

\maketitle


\vspace{-1.5em}

\begin{abstract}We establish two geometric inequalities, respectively, for  harmonic functions in exterior Dirichlet problems, and for  Green's functions in interior Dirichlet problems, where the boundary surfaces are smooth and convex. Both inequalities involve integrals over the mean curvature and the Gaussian curvature on an equipotential surface, and the normal derivative of the harmonic potential thereupon. These inequalities generalize a geometric conservation law for equipotential curves in  dimension two, and offer solutions to two free boundary problems in three-dimensional electrostatics.  \end{abstract}

\pagenumbering{roman}


\pagenumbering{arabic}

\section{Introduction}

Consider  a three-dimensional exterior Dirichlet problem (``3-exD'' below),  where a non-constant  harmonic function $ U(\bm r),\bm r\in\Omega\subset \mathbb R^3$ solves a Laplace equation \begin{align}
\nabla^2 U(\bm r)=0,\quad \bm r\in\Omega\label{eq:3D_Laplace}
\end{align}in an unbounded domain $ \Omega$, whose boundary $ \partial \Omega$ is a smooth and connected surface, on which $U(\bm r) $ remains constant.  The flux condition \begin{align}
-\oint_{\partial\Omega}\bm n\cdot\nabla U(\bm r)\D S=\Phi>0\label{eq:3D_flux}
\end{align}  (with   $\bm n$  being the outward unit normal on $ \partial \Omega$, and $ \D S$ the surface element) is equivalent to  the following asymptotic behavior:\begin{align}
 U(\bm r)\sim \frac{\Phi }{4\pi|\bm r|},\quad |\bm r|\to+\infty.\label{eq:U_inf}
\end{align}   If $ \mathbf 0\notin\Omega\cup\partial\Omega$, then one can  define the Green's function $ G(\bm r)=G_{D}^{\partial\Omega}(\mathbf 0,\bm r)$ in  three-dimensional interior Dirichlet problem (``3-inD'' below) as the solution to \begin{align}
\left\{\begin{array}{ll}
\nabla^2 G(\bm r)=0, & \bm r\in\mathbb R^3\smallsetminus(\Omega\cup\partial\Omega\cup\{\mathbf 0\}),\\
G(\bm r)=\text{0,} & \bm r\in \partial\Omega, \\\displaystyle-\lim_{\varepsilon\to0^+}\oint_{|\bm r|=\varepsilon}\bm n\cdot\nabla G(\bm r)\D S=1.
\end{array}\right.\label{eq:3-inD_G}
\end{align} According to the maximum principle for harmonic functions, we have $ U(\bm r)>0,\bm r\in\Omega\cup\partial\Omega $ in 3-exD and  $ G(\bm r)>0,\bm r\in\mathbb R^3\smallsetminus(\Omega\cup\partial\Omega\cup\{\mathbf 0\}) $ in 3-inD. In what follows, we write $ \Sigma_\varphi$ for the equipotential surface on which the harmonic function [either $U(\bm r)$ in 3-exD or $G(\bm r)$ in 3-inD] equals a given non-negative $\varphi$.

In classical physics, the 3-exD (resp.~3-inD) problem occurs in electrostatic equilibrium of an isolated metallic conductor (resp.~a point charge enclosed in a metallic cavity), where our harmonic function of interest is the electrostatic potential, and   $ E(\bm r)=|\nabla U(\bm r)|$ (resp.~$ E(\bm r)=|\nabla G(\bm r)|$) is the magnitude of the electrostatic field, also known as ``field intensity''. If the boundary surface     $ \partial \Omega$ is smooth and  convex (with non-negative Gaussian curvature $ K(\bm r)\geq0,\bm r\in\partial\Omega$), then we have $ E(\bm r)\neq0$ in both 3-exD and 3-inD problems \cite[Proposition 3.2]{MaZhang2014}, and all the equipotential surfaces (excluding the boundary) are smooth and strictly convex (with positive Gaussian curvature $ K(\bm r)>0,\bm r\in\Sigma_\varphi\neq\partial\Omega$) \cite[Theorem 1.1]{JostMaOu2012}.

A quantitative understanding of the interplay between geometry (shape of an equipotential surface $\partial \Omega $) and physics (the distribution of field intensity  $ |\nabla U(\bm r)|,\bm r\in\partial\Omega$) has practical consequences, ranging from the design of lightning-rods  \cite{PriceCrowley1985} to the self-assembly of metallic nanoparticles \cite{Kalsin2006}. The ``common knowledge'' that strongest field  accompanies greatest curvature  is mathematically unfounded \cite[Fig.~5]{PriceCrowley1985}. Therefore, instead of following \textit{electricians' folklore} about  pointwise causal relationship between curvature and field intensity, it is more sensible to study statistical correlations between geometric and physical quantities, in a non-local manner.

In this work, we focus on  3-exD and 3-inD problems with smooth and  convex boundaries (``3-exDc'' and ``3-inDc'' hereafter), and investigate  integrals on equipotential surfaces $ \Sigma_\varphi$ with bounded mean curvature\footnote{By choosing an outward unit normal vector, we are adopting a sign convention where the unit sphere has  mean curvature $H=-1$.} $H(\bm r)\leq0,\bm r\in\Sigma_\varphi $, Gaussian curvature  $K(\bm r)\geq0,\bm r\in\Sigma_\varphi $, and non-vanishing field intensity   $E(\bm r)\neq0,\bm r\in\Sigma_\varphi $. (For convenience, we shall also use the term  ``electrostatic problems''  to cover both   3-exDc and 3-inDc.) We will construct inequalities for surface integrals involving $ H(\bm r)$, $K(\bm r)$ and $ E(\bm r)$, thereby presenting \textit{a priori} bounds for  statistical averages of field intensity fluctuation $ |\bm n\times\nabla\log E(\bm r) |^{2}$ through statistical averages of   of curvature fluctuation $ H^2(\bm r)-K(\bm r)$.

After laying out the  geometric settings in \S\ref{sec:geo_prep}, we will prove our main result   (Theorem \ref{thm:quasi-csv}) and its consequence (Corollary \ref{cor:free_bd}) in \S\ref{sec:3Dconv_ineq}.

\begin{theorem}[Geometric inequalities on convex equipotential surfaces]\label{thm:quasi-csv}For every level set $ \Sigma$ in 3-exDc, we have the following inequality (strict unless $ \partial\Omega $ is a sphere):\begin{align}
\oint_{\Sigma}\frac{4[H^2(\bm r)-K(\bm r)]-|\bm n\times\nabla\log |\nabla U(\bm r)| |^{2}}{|\nabla U(\bm r)|}\D S\geq0.\label{eq:3-exDc}
\end{align}   For every level set $ \Sigma$ in 3-inDc, we have  following inequality  (strict unless $ \partial\Omega $ is a sphere centered at the origin):\begin{align}
\oint_{\Sigma}\frac{4[H^2(\bm r)-K(\bm r)]-|\bm n\times\nabla\log |\nabla G(\bm r)| |^{2}}{|\nabla G(\bm r)|}\D S\leq0.\label{eq:3-inDc}
\end{align} \end{theorem}

\begin{corollary}[Spherical solutions to two free boundary value problems]\label{cor:free_bd}If there is a spherical equipotential surface in 3-exDc, then the boundary  $ \partial\Omega $ must be a sphere. If there is an equipotential surface in 3-inDc on which $ |\nabla G(\bm r)|$ remains constant, then  $ \partial\Omega $ must be a sphere centered at the origin.\end{corollary}

 Two-dimensional analogs of electrostatic problems can be regarded as the situations of three-di\-men\-sion\-al cylindrical surfaces  with translational invariance along the $z$-axis. For the two-dimensional cross-section of such cylindrical surfaces, the curvature of an equipotential curve  becomes $ \kappa=-2H $, while the Gaussian curvature vanishes identically $ K\equiv0$. Therefore, the surface integrals appearing in Theorem~\ref{thm:quasi-csv} are reminiscent of  the following integrals on equipotential curves \cite[(1.14)]{Geom2D}:\begin{align}
\oint_\Sigma \frac{[\kappa(\bm r)]^{2}-|\bm n\times\nabla\log |\nabla U(\bm r)| |^{2}}{|\nabla U(\bm r)|}\D s\quad\text{and}\quad \oint_\Sigma \frac{[\kappa(\bm r)]^{2}-|\bm n\times\nabla\log |\nabla G(\bm r)| |^{2}}{|\nabla G(\bm r)|}\D s\label{eq:2D_csv}
\end{align}   for 2-exD and 2-inD, respectively. In our previous work \cite[\S2.2 and \S3]{Geom2D},
 we have shown that both integrals are constants (independent of $\varphi$) when the boundary $ \partial \Omega$ is a smooth Jordan curve.
Our proof in   \S\ref{sec:3Dconv_ineq} will reveal a unified mechanism underlying the geometric inequalities in Theorem~\ref{thm:quasi-csv} and the geometric conservation laws in  \eqref{eq:2D_csv}.

Theorem~\ref{thm:quasi-csv} unveils a subtle constraint between  the fluctuations of curvatures and field intensity  on a single equipotential surface. Its toy application (Corollary \ref{cor:free_bd}), by contrast, contains less surprising statements (cf.~stronger results  for the Green's functions in \cite[Theorem III.2]{PayneSchaefer1989}). To conclude this article, we will strengthen the first half of  Corollary \ref{cor:free_bd} in $ \mathbb R^d$ ($ d\geq2$), and sharpen its second half in $ \mathbb R^2$.

\section{Geometric Preparations\label{sec:geo_prep}}In this section, we set up a geometric framework for  electrostatic problems (\S\ref{subsec:curv_coord}), and prepare some differential formulae (\S\ref{subsec:Lap_n_H}) that will be useful later. Unavoidably, we will recover some standard identities in classical differential geometry \cite{Minkowski1903,Weatherburn1,doCarmo2016}, as well as reproduce part of the modern investigations of level sets for Green's functions on manifolds \cite{Colding2012,ColdingMinicozzi2013,ColdingMinicozzi2014,MaZhang2014}. Nevertheless, we choose to include our derivations here, for the sake of consistency and accessibility. Indeed, the availability of certain vector calculus identities in the flat Euclidean space  $ \mathbb R^3$ does make our computations more straightforward than generic cases on intrinsically curved Riemannian manifolds.

\subsection{Curvilinear coordinates and Laplacian decomposition\label{subsec:curv_coord}}
Akin to  our previous work \cite[\S2.1]{Geom2D}, we set up a curvilinear coordinate system $ \bm r(\varphi ,u,v)\equiv\bm r(u^0,u^1,u^2)$ that is compatible with equipotential surfaces in $ \mathbb R^3$. In this coordinate system,  $ \varphi\equiv u^0$ coincides with the value of the harmonic potential [$U(\bm r)$ in 3-exD, $ G(\bm r)$ in 3-inD], and a pair of points on distinct equipotential surfaces share the same $(u,v)\equiv(u^1,u^2) $  coordinates if and only if they are joined by an integral curve of $ \nabla \varphi$.
Thus, a family of  equipotential surfaces $\Sigma_\varphi $  evolve according to the following equation\begin{align}\frac{\partial\bm r(\varphi ,u,v)}{\partial\varphi}=-\frac{\bm n}{E(\bm r)},\quad\bm r\in\Sigma_\varphi,\label{eq:GM_flow}\end{align} which conserves the total surface flux \begin{align}\frac{\D}{\D\varphi}\oint_{\Sigma_\varphi}\bm n\cdot\nabla\varphi\D S=0.\end{align}This conservation is expected from  the Gau{\ss} law of electrostatics, which is part of the Maxwell equations for classical electrodynamics \cite[\S1.4, \S1.7]{Jackson:EM}. Hereafter, we will refer to \eqref{eq:GM_flow} as the Gau{\ss}--Maxwell flow.

On each equipotential surface, we define the components of  the \textit{covariant metric tensor} $ (g_{ij})$  as $g_{ij}:=\partial_i\bm r\cdot\partial_j\bm r $, where $\partial_i$
is short-hand for $\partial/\partial u^i$. The \textit{contravariant metric tensor}
$(g^{ij})$ is  the matrix inverse of $(g_{ij})$. The line element on each equipotential surfaces is given by  $ \D s^2=g_{ij}\D u^i\D u^j$, where the Einstein summation convention is applied hereinafter, and a Latin index takes values in $ \{1,2\}$.

On each equipotential surface, we have the {Gauss} formula \cite[\S4.3]{doCarmo2016}: $\partial_i\partial_j\bm
 r=\Gamma_{ij}^k\partial_k\bm r+b_{ij}\bm n $, for \textit{connection coefficients}   $\Gamma_{ij}^k:=g^{k\ell}\partial_i\partial_j\bm r\cdot\partial_\ell\bm r=\frac12g^{k\ell}(\partial_ig_{\ell
j}+\partial_jg_{i\ell}-\partial_\ell g_{ij})$  \cite[\S5.7]{doCarmo2016} and the \textit{coefficients of second fundamental form}   $b_{ij}:=\partial_i\partial_j\bm r\cdot\bm n$.
The components of the Weingarten transform  $\hat W=(b_i^j)$ is defined by $b_i^j:=g^{jk}b_{ki}$
and appears in the Weingarten formula: $\partial_i\bm n=-b_i^j\partial_j\bm
r $, that is, $\D \bm n=-\hat W \D\bm r$ for infinitesimal changes tangent to the equipotential surface.
The \textit{mean curvature} is half the trace of the Weingarten transform: $H:=\frac12\Tr(\hat W)=\frac12(b_1^1+b_2^2)=\frac12g^{ij}b_{ij};$
while the \textit{Gaussian curvature} is the determinant of the Weingarten transform: $K:=\det(\hat W)=b_1^1b_2^2-b_2^1b_1^2$.

Being compatible with the Gau{\ss}--Maxwell flow equation in \eqref{eq:GM_flow}, we have $ g_{00}:=\partial_0\bm r\cdot\partial_0\bm r=E^{-2}=1/g^{00}$ and  $g_{0i}=g^{0i}:=\partial_0\bm r\cdot\partial_i\bm r=0$. In this way, the Euclidean line element $ \D \uline{s}^2=\D x^2+\D y^2+\D z^2$  can be reformulated as \begin{align}
\D \uline{s}^2=g_{\mu\nu}\D u^\mu\D u^\nu=\frac{\D \varphi^2}{E^2}+g_{ij}\D u^i\D u^j,
\end{align}where a Greek index takes values in $ \{0,1,2\}$. One may extend
the definition of connection coefficients as $\partial_\mu\partial_\nu\bm
 r=\Gamma_{\mu\nu}^\lambda\partial_\lambda\bm r$, where the newly-arisen connection coefficients will be computed in the following proposition.
\begin{proposition}[Connection coefficients]\label{prop:conn0} We have the following computations for connection coefficients involving the index $0$:\begin{align}&\Gamma _{ij}^{0} =-Eb_{ij},\quad\Gamma _{j0}^{0} ={}-\frac{1}{E}\frac{\partial E}{\partial u^{j}}, \quad
\Gamma _{j0}^{k} =\frac{b_{j}^{k}}{E};\label{eq:Gamma1}\\&\Gamma _{ij}^{0} =-\frac{E^2}{2}\frac{\partial g_{ij}}{\partial \varphi },\quad\Gamma _{00}^{j}=-\frac{1}{2}g^{jm}\frac{\partial g_{00}}{\partial u^{m}}=%
\frac{1}{E^{3}}g^{jm}\frac{\partial E}{\partial u^{m}};\label{eq:Gamma2}\\&\Gamma^0_{00}=\frac{1}{2}g^{00}\partial_0g_{00}=-\frac{\partial}{\partial\varphi}\log E=\Gamma _{m0}^m.\label{eq:Gamma3}\end{align}
\end{proposition} \begin{proof}
 To prove the three identities in \eqref{eq:Gamma1}, it would suffice to compare the equation $\partial_\mu\partial_\nu\bm
 r=\Gamma_{\mu\nu}^\lambda\partial_\lambda\bm r$ with the
 Gau{\ss}  and Weingarten formulae: \begin{align}\frac{\partial ^{2}{\bm r}}{\partial u^{i}\partial u^{j}} =\Gamma _{ij}^{k}%
\frac{\partial {\bm r}}{\partial u^{k}}-Eb_{ij}\frac{\partial {\bm r}}{%
\partial \varphi }, \qquad
\frac{\partial }{\partial u^{j}}\left( E\frac{\partial {\bm r}}{\partial
\varphi }\right) =b_{j}^{k}\frac{\partial {\bm r}}{\partial u^{k}}.\end{align}
The two identities in \eqref{eq:Gamma2} follow from the
Christoffel formula  $\Gamma_{\mu\nu}^\lambda=\frac12g^{\lambda\eta}(\partial_\mu
g_{\eta
\nu}+\partial_\nu g_{\mu\eta}-\partial_\eta g_{\mu\nu})$.

Before deducing \eqref{eq:Gamma3}, we compare the two expressions of  $\Gamma_{ij}^0$ in \eqref{eq:Gamma1} and \eqref{eq:Gamma2} and write down\begin{align}\label{eq:4HE}\frac{\partial g_{ij}}{\partial \varphi
}=\frac2Eb_{ij},\quad\frac{\partial\log\det( g_{ij})}{\partial \varphi
}=g^{ij}\frac{\partial g_{ij}}{\partial \varphi
}=\frac2Eg^{ij}b_{ij}=\frac{4H}{E}.\end{align}On the other hand, the
Laplace equation  $ \nabla^2\varphi(\bm r)=-\nabla\cdot\bm E(\bm r)=0 $  implies zero divergence of
$\bm E$-field, i.e. $\partial_0\log(E\sqrt{g})
=0$, (hereafter $g=\det(g_{ij}) $), thus \eqref{eq:4HE} gives rise to
\begin{align}2H+\frac{\partial E}{\partial \varphi
}=0,\qquad\text{i.e.}\quad \bm n\cdot\nabla\log E=2H.\label{eq:2H}\end{align}It is easy to recast \eqref{eq:2H} into the harmonic coordinate condition
$\Gamma^0:=g^{ij}\Gamma_{ij}^0+g^{00}\Gamma_{00}^0=0 $, which leads to  \eqref{eq:Gamma3}.
\end{proof}\begin{remark}Using the identity $\partial_0g_{ij}=2b_{ij}/E $, we can  also readily  deduce $\partial_0 g^{ij}=-2g^{ik}b_k^j/E$.
\eor\end{remark}

The three-dimensional Laplace operator $ \Delta=\partial_x^2+\partial_y^2+\partial_z^2$ can be presented in curvilinear coordinates
as \begin{align}\Delta=g^{\mu\nu}(\partial_\mu\partial_\nu-\Gamma_{\mu\nu}^\lambda\partial_\lambda)=\frac1{\sqrt{\det(g_{\mu\nu})}}\partial_\lambda\left( g^{\lambda\eta}\sqrt{\det(g_{\mu\nu})}\partial_\eta \right).\end{align}
Here, $\det(g_{\mu\nu})=g/E^{2}$ for $ g=\det(g_{ij})$. Similarly, one can define the \textit{Laplace operator on equipotential surface $\Sigma$}
as \begin{align}\Delta _{\Sigma }=g^{ij}( \partial_i\partial_j-\Gamma_{ij}^k\partial_k)=\frac1{\sqrt{g}}\partial_k\left( g^{k\ell}\sqrt{g}\partial_\ell \right).\end{align}\begin{proposition}[Decomposition of Laplacian]\label{prop:Lap_decomp}The Laplace operator $ \Delta$ can be rewritten as\begin{align}\Delta =\Delta _{\Sigma }+E^{2}\frac{\partial ^{2}}{\partial \varphi ^{2}}-\frac{1%
}{E}g^{jm}\frac{\partial E}{\partial u^{m}}\frac{\partial }{\partial u^{j}}.\label{eq:LapOp}\end{align}\end{proposition}\begin{proof}By definition, we have\begin{align}\Delta =\Delta _{\Sigma }-g^{ij}\Gamma _{ij}^{0}\frac{\partial }{\partial
\varphi }+g^{00}\left( \frac{\partial ^{2}}{\partial \varphi ^{2}}-\Gamma
_{00}^{j}\frac{\partial }{\partial u^{j}}-\Gamma _{00}^{0}\frac{\partial }{%
\partial \varphi }\right).\end{align}With the substitution of $g^{00}=E^2$ and the expressions for  $\Gamma _{ij}^0,\Gamma^j_{00},\Gamma^0_{00} $ from Proposition~\ref{prop:conn0}, we obtain the claimed result in \eqref{eq:LapOp}. \end{proof}\begin{corollary}[Geometric description of $ \nabla\log E$]We have \begin{align}\label{eq:gradlnE}\nabla\log
E(\bm r)=k(\bm r)\bm N(\bm r)+2H(\bm r)\bm n(\bm r).\end{align} Here, $k(\bm r)$ is
the curvature (inverse of the radius of curvature) of the electric field
line ($\bm E$-line) that passes $\bm r$, and $H(\bm r)$ is the mean curvature
of the equipotential surface that passes $\bm r$, with $\bm N$ and $\bm n$
being the respective unit normal vectors for the $\bm E$-line and equipotential surface. \end{corollary}\begin{proof}As we already have the normal derivative $\bm n\cdot\nabla\log E=2H $ in  \eqref{eq:2H}, it is sufficient to show that the tangential gradient $[\bm n\times \nabla\log
E(\bm r)]\times\bm n=g^{jm}(\partial_m\log E)\partial_j\bm r$ is equal to $ k\bm N$. To fulfill this task, we compute\begin{align}\label{eq:E_line_k}\mathbf 0={}&\Delta\bm r=\Delta _{\Sigma }\bm r+E^{2}\frac{\partial ^{2}\bm r}{\partial \varphi ^{2}}-\frac{1%
}{E}g^{jm}\frac{\partial E}{\partial u^{m}}\frac{\partial\bm r }{\partial u^{j}}\notag\\={}&2H\bm n-E^2\frac{\partial(\bm n/E)}{\partial \varphi}-g^{jm}\frac{\partial\log E}{\partial u^{m}}\frac{\partial\bm r }{\partial u^{j}}=-E\frac{\partial\bm n}{\partial \varphi}-g^{jm}\frac{\partial\log E}{\partial u^{m}}\frac{\partial\bm r }{\partial u^{j}}, \end{align} where the definition for the curvature of a curve $ k\bm N=-E\partial_0\bm n$ can be substituted in the last step.\end{proof}\begin{remark}The result in \eqref{eq:gradlnE} is well known in physics, as the tangential and normal components of $\nabla\log E $ can be easily derived from elementary vector analysis \cite[p.~591]{Jackson:EM}
and the Gau{\ss} theorem of electrostatic field \cite[p.~52, Problem 1.11]{Jackson:EM},
respectively. We have rederived  \eqref{eq:gradlnE} in our curvilinear coordinate system as a double check of the computations involving the connection coefficients and the Laplacian.

Later on, we will often use the notation $Df:=g^{ij}\partial_if\partial_j\bm r $ for the tangential gradient of a smooth function $f$. This allows us to abbreviate \eqref{eq:E_line_k} as $\partial_\varphi\bm n=D(1/E)$ for the Gau{\ss}--Maxwell flow. It follows immediately from \eqref{eq:E_line_k}  that \begin{align}
\partial_{\varphi}\bm E=\partial_\varphi(E\bm n)=ED(1/E)-2H\bm n=-\nabla\log E.\label{eq:d_phi_E_vec}
\end{align}It is also easy to verify, for the Gau{\ss}--Maxwell flow,  that the following commutation relation holds:\begin{align}\label{eq:comm_rel}
\partial _{\varphi}(Df)-D(\partial_{\varphi}f)=-\frac{\hat W Df}{E}-\bm n(Df)\cdot \left(D\frac{1}{E}\right):=-\frac{ g^{ik}b^j_k\partial_if\partial_j\bm r}{E}-\bm ng^{ij}\partial_if\partial_j\frac1E.
\end{align}In particular,  \eqref{eq:E_line_k} and the commutation relation  above would entail\begin{align} \label{eq:Lap_diff_n}\Delta\bm n-\Delta_\Sigma\bm n=E^2\partial_\varphi [D(1/E)]+\hat W D\log E=2DH+2(\hat W-2H) D\log E-\bm n|D\log E|^{2},\end{align}a formula that will be used later in \S\ref{subsec:Lap_n_H}.\eor\end{remark}\subsection{Evolution of mean and Gaussian curvatures on  equipotential surfaces\label{subsec:Lap_n_H}} Since we will be interested in tracking down the changes of curvatures across different equipotential surfaces, it is sensible to derive formulae for the the derivatives of curvatures with respect to the potential $\varphi$.  \begin{proposition}[Evolution of the second fundamental form]\label{prop:2nd_fund_ev}We have the following identities
\begin{align}\frac{\partial b_{ij}}{\partial \varphi }=( b_{j}^{k}b_{ki}- \partial_i\partial_j+\Gamma_{ij}^k\partial_k) \frac{1}{E}\label{eq:bij}\end{align}and\begin{align}2\frac{\partial
H}{\partial\varphi}=-\Delta_\Sigma\frac1E-\frac{4H^2-2K}{E}.\label{eq:Hphi}\end{align}
\end{proposition}\begin{proof}From the identity $
\partial_0(\partial_i\partial_j\bm
r)=\partial_i(\partial_0\partial_j\bm r)$, we may deduce   $\partial_0 \Gamma _{ij}^{0}+\Gamma _{ij}^{\nu }\Gamma _{\nu
0}^{0}={\partial_i \Gamma _{j0}^{0}}+\Gamma _{j0}^{\nu }\Gamma
_{\nu i}^0$. This results in \eqref{eq:bij}, upon substitution of the connection coefficients. Combining $2H=g^{ij}b_{ij}$ and  $\partial_0 g^{ij}=-2g^{ik}b_k^j/E$
with \eqref{eq:bij}, we obtain\begin{equation*}2\frac{\partial
H}{\partial\varphi}=-\Delta_\Sigma\frac1E-\frac{b_j^kb_k^j}{E},\quad\text{where  } b_j^kb_k^j=\Tr(\hat W^2)=4H^2-2K . \end{equation*}This verifies \eqref{eq:Hphi}.\end{proof}\begin{proposition}[Evolution of Gaussian curvature]\label{prop:GC_ev}We have the following formula \begin{align}\label{eq:Kphi1}\frac{\partial}{\partial\varphi}(K\sqrt{g})=-\partial_i\left( \beta^{ij}\sqrt{g}\partial_j\frac{1}{E}\right)\Longleftrightarrow\frac{\partial}{\partial\varphi}\frac{K}{E}=-\frac{1}{E\sqrt{g}}\partial_i\left( \beta^{ij}\sqrt{g}\partial_j\frac{1}{E}\right)\end{align}for $\beta^{ij}:=2Hg^{ij}-g^{ik}b^j_k=\crb^{ij}/g$,  where $(\crb^{ij})=\left(\begin{smallmatrix}b_{22} & -b_{12} \\
-b_{12} & b_{11} \\
\end{smallmatrix}\right)$ is the adjugate matrix of $(b_{ij}) $. \end{proposition}\begin{proof}Using Jacobi's formula for the derivative of a determinant, we may verify that\begin{align}\label{eq:Kphi2}\frac{\partial}{\partial\varphi}(K\sqrt{g})={}&\frac{\partial}{\partial\varphi}\left(\frac{\det(b_{ij})}{\sqrt{g}}\right)=\frac{\crb^{ij}}{\sqrt{g}}\frac{\partial b_{ij}}{\partial\varphi}-\frac{2HK\sqrt{g}}{E}\notag\\={}&\frac{\crb^{ij}}{\sqrt{g}}( b_{j}^{k}b_{ki}- \partial_i\partial_j+\Gamma_{ij}^m\partial_m) \frac{1}{E}-\frac{2HK\sqrt{g}}{E}=\beta^{ij}\sqrt{g}( - \partial_i\partial_j+\Gamma_{ij}^m\partial_m) \frac{1}{E},\end{align}where we have quoted \eqref{eq:bij} in the penultimate step, before using the  relation $ \crb^{ij} b_{j}^{k}b_{ki}/g=K\delta^{j}_kb^k_j=2HK$ in the last step. Then, we note that the  Codazzi--Mainardi equation $\partial_kb_{ij}-\partial_jb_{ik}+\Gamma^\ell_{ij}b_{\ell k}-\Gamma^\ell_{ik}b_{\ell j}=0$ and the vanishing covariant derivatives of the metric  $ (g^{ik})_{;\ell}:=\partial_\ell g^{ik}+g^{im}\Gamma^k_{m\ell}+g^{km}\Gamma^i_{m\ell}=0$ allow us to compute $\partial_i b^k_j+\Gamma^k_{i\ell}b^\ell_j-\Gamma^\ell_{ij}b^k_\ell=:b^k_{j;i}=b^k_{i;j} $ and $ (\beta^{ik})_{;k}=b^\ell_{\ell;k} g^{ik}-g^{ij}b^k_{j;k}=b^\ell_{\ell;k} g^{ik}-g^{ij}b^k_{k;j}=0$, thereby leading to \begin{align}-\partial_i\left( \beta^{ij}\sqrt{g}\partial_jf\right)={}&\beta^{ij}\sqrt{g}( - \partial_i\partial_j+\Gamma_{ij}^m\partial_m) f,\label{eq:Kphi3} \end{align}for every smooth function $f$. Combining the results in \eqref{eq:Kphi2} and \eqref{eq:Kphi3}, we arrive at the claimed formula in \eqref{eq:Kphi1}.
\end{proof}\begin{remark}When the field intensity $E$ is non-vanishing on an entire equipotential surface $\Sigma_\varphi $, we may double-check the reasonability of \eqref{eq:Kphi1} by the following computation\begin{align}\label{eq:GB0}\frac{\D}{\D \varphi}\oint_{\Sigma_\varphi}K\D S=\oint_{\Sigma_\varphi}\frac{1}{\sqrt{g}}\frac{\partial}{\partial\varphi}(K\sqrt{g})\D S=-\oint_{\Sigma_\varphi}\frac{1}{\sqrt{g}}\partial_i\left( \beta^{ij}\sqrt{g}\partial_j\frac{1}{E}\right)\D S=0.\end{align} On the other hand, we know from the Gau{\ss}--Bonnet theorem that $ \oint_{\Sigma_\varphi}K\D S=2\pi\chi(\Sigma_\varphi)$, where $ \chi(\Sigma_\varphi)$ is the Euler--Poincar\'e characteristic that determines the topology of $ \Sigma_\varphi$. The result in \eqref{eq:GB0} is thus expected from the non-critical $\bm E $-lines that establish diffeomorphisms among all the equipotential surfaces in a neighborhood of  $ \Sigma_\varphi$. \eor\end{remark}
\begin{corollary}[Weatherburn formula {\cite[p. 231]{Weatherburn1}}]The following identity holds on every smooth surface  \begin{align}\Delta_\Sigma\bm n=(2K-4H^2)\bm n-2(\bm n\times\nabla H)\times\bm n.\label{eq:surf_Lap_n}\end{align}   \end{corollary}
\begin{proof}Applying \eqref{eq:Kphi3} to the three Euclidean components of $ \bm r=x\bm e_x+y\bm e_y+z\bm e_z$, we can quickly  recover the following formula of Minkowski~\cite{Minkowski1903}:\begin{align}\label{eq:MinTech}\frac{1}{\sqrt{g}}\partial_i\left( \beta^{ij}\sqrt{g}\partial_j\bm r\right)=2K\bm n,\end{align}with the computation\begin{align}\frac{1}{\sqrt{g}}\partial_i\left( \beta^{ij}\sqrt{g}\partial_j\bm r\right)=\beta^{ij}(  \partial_i\partial_j-\Gamma_{ij}^m\partial_m)\bm r=\frac{\crb^{ij}b_{ij}\bm n}{g}=\frac{2\det(b_{ij})}{g}\bm n=2K\bm n.\end{align}This in turn allows us to verify the   Weatherburn formula via\begin{align}\Delta_\Sigma\bm n-2K\boldsymbol n=\frac{1}{\sqrt{g}}\partial_i\left( g^{ij}\sqrt{g}\partial_j\bm n-\beta^{ij}\sqrt{g}\partial_j\bm r\right)=-\frac{1}{\sqrt{g}}\partial_i( 2Hg^{ij}\sqrt{g}\partial_j\bm r)=-4H^2\bm n-2DH,\end{align}where we have exploited $ g^{ij}\sqrt{g}\partial_j\bm n=-g^{ik}b^j_k\sqrt{g}\partial_j\bm r$, $\beta^{ij}:=2Hg^{ij}-g^{ik}b^j_k $ and  a familiar relation $ \Delta_\Sigma\bm r=\frac{1}{\sqrt g}\partial_i(g^{ij}\sqrt{g}\partial_j\bm r)=2H\bm n$.\end{proof}

 Combining  \eqref{eq:Lap_diff_n} and \eqref{eq:surf_Lap_n}, we immediately arrive at the  following representation of $ \Delta\bm n:=\bm e_x\Delta(\bm n\cdot\bm e_x)+\bm e_y\Delta(\bm n\cdot\bm e_y)+\bm e_z\Delta(\bm n\cdot\bm e_z)$:\begin{align}
\Delta\bm n=2(\hat W-2H) D\log E-\bm n(|D\log E|^{2}+4H^2-2K),\label{eq:Lap_n}
\end{align}a result that will be used later in Corollary~\ref{cor:Lap_n_E}.

\section{Main result and applications\label{sec:3Dconv_ineq}}Like previous studies of level sets for harmonic functions \cite{Colding2012,ColdingMinicozzi2013,ColdingMinicozzi2014,MaZhang2014,Geom2D}, we will build a monotonicity result (\S\ref{subsec:mono_int}) on positive definite quadratic forms, before subsequently applying it to the proof of Theorem \ref{thm:quasi-csv} and Corollary \ref{cor:free_bd} in \S\ref{subsec:geo_ineq}. We will finally devote \S\ref{subsec:rel_prob} to some generalizations of Corollary \ref{cor:free_bd}.
\subsection{Monotonicity of an integral on  equipotential surface\label{subsec:mono_int}}To prepare for the proof of  Theorem~\ref{thm:quasi-csv}, we compute  two more quantities: $ \Delta\log E$ and $ \Delta\frac{\bm n}{E}$. Both these quantities vanish in two-dimensional electrostatic problems, as one can easily check by complex analytic techniques. \begin{proposition}[Laplacian representation of Gaussian curvature]\label{prop:lnE_K}There is a geometric  identity \begin{align}\label{eq:dlnE2K_id} \Delta\log E+2K=0,\end{align} which is a special case of \cite[Proposition 1.4]{Zhou2013LevelSet}.\end{proposition}\begin{proof}We  first  employ  \eqref{eq:Hphi} to compute\begin{align}-\frac{\partial^2}{\partial\varphi^2}\log
E=\frac\partial{\partial\varphi}\left(\frac{2H}{E}\right)=\frac{4H^2}{E^2}+\frac2E\frac{\partial
H}{\partial\varphi}=-\frac1E\Delta_\Sigma\frac1E+\frac{2K}{E^2}.\label{eq:d2lnE}\end{align}Meanwhile, we may use the definition of Laplacian $\Delta$ in the curvilinear coordinate system to evaluate\begin{align}
 \Delta\log
E=\frac{E}{\sqrt{g}}{\partial_\mu}\left( \frac{\sqrt{g}}{E} g^{\mu\nu}\partial_\nu\log E\right)=-\frac{E}{\sqrt{g}}{\partial_\mu}\left({\sqrt{g}} g^{\mu\nu}\partial_\nu\frac1 E\right)=-E\Delta_\Sigma\frac1E+E^2\frac{\partial^2}{\partial\varphi^2}\log
E.\label{eq:DeltalnE}
\end{align}Combining \eqref{eq:d2lnE} with \eqref{eq:DeltalnE}, we arrive at the claimed identity.
\end{proof}
\begin{corollary}[A geometric representation of $ \Delta\frac{\bm n}{E}$]\label{cor:Lap_n_E}We have the following formula:\begin{align}
\Delta\frac{\bm n}{E}=4\left(\beta^{ij}\partial_i\frac{1}{E}\partial_j\bm r+\frac{K\bm n}{E}\right).
\end{align}\end{corollary}\begin{proof}Combining our formula for $ \Delta\bm n$ in \eqref{eq:Lap_n} with the  identity    $ \Delta\log E+2K=0$, and noting that  $ (\nabla\log E\cdot\nabla)\bm n=2H(\bm n\cdot\nabla)\bm n+(D\log E\cdot\nabla)\bm n=2HD\log E-\hat WD\log E$, we can    compute \begin{align}\Delta\frac{\bm n}{E}={}&\frac{\Delta\bm n}{E}+\bm n\Delta\frac{1}{E}+2g^{\mu\nu}\partial_\mu\frac{1}{E}\partial_\nu\bm n\notag\\={}&2\beta^{ij}\partial_i\frac{1}{E}\partial_j\bm r-\frac{(4H^2-2K+|D\log E|^{2})\bm n}{E}-\bm n\nabla\cdot\left( \frac{1}{E}\nabla\log E\right)-\frac{2}{E}(\nabla\log E\cdot\nabla)\bm n\notag\\={}&2\beta^{ij}\partial_i\frac{1}{E}\partial_j\bm r+\frac{4K\bm n}{E}-\frac{2}{E}(2H-\hat W)D\log E=4\left(\beta^{ij}\partial_i\frac{1}{E}\partial_j\bm r+\frac{K\bm n}{E}\right),\label{eq:Lap_nE}\end{align}as claimed.
\end{proof}

\begin{corollary}[Evolution of a surface integral]\label{cor:betaij_int}We have the following derivative formula:\begin{align}
\frac{\D}{\D\varphi}\oint_{\Sigma_{\varphi}}\left( H^2 -K-\frac{|D\log E|^2}{4}\right)\frac{\D S}{E}=-\frac{3}{2}\oint_{\Sigma_\varphi}\beta^{ij}\partial_{i}\frac{1}{E}\partial_{j}\frac{1}{E}\D S.\label{eq:W1'-F1'}
\end{align}\end{corollary}\begin{proof}One can verify \eqref{eq:W1'-F1'} by a brute-force computation, using the derivatives of $ H$, $K$, $E$ and $g^{ij}$ studied in \S\ref{sec:geo_prep}.
Here, we will build our proof on  Proposition \ref{prop:lnE_K} and Corollary  \ref{cor:Lap_n_E}, to highlight the mechanism shared by the  derivative formula   \eqref{eq:W1'-F1'}  for three-dimensional electrostatics and its two-dimensional counterpart \cite[(2.22)]{Geom2D}.

Specializing the vector Green identity \cite[p.~156]{ColtonKress}\begin{align}\label{eq:vec_Green_id_3D}\int_{\mathfrak D}[\bm Q\cdot\Delta\bm F-\bm F\cdot\Delta\bm Q]\D ^3\bm r=\oint_{\partial \mathfrak D}[(\bm \nu\times\bm Q)\cdot(\nabla\times\bm F)+(\bm \nu\cdot\bm Q)(\nabla\cdot\bm F)-(\bm \nu\times\bm  F)\cdot(\nabla\times\bm Q)-(\bm \nu\cdot\bm F)(\nabla\cdot\bm Q)]\D S\end{align}to $ \bm Q(\bm r)=\nabla\log E(\bm r)$ and $\bm F(\bm r)=\bm n(\bm r)/E(\bm r) $, we may  put down\begin{align}{}&\int_\mathfrak D\left[\nabla\log E(\bm r)\cdot\Delta\frac{\bm n(\bm r)}{E(\bm r)}-\frac{\bm n(\bm r)}{E(\bm r)}\cdot\Delta\nabla\log E(\bm r)\right]\D ^3\bm r\notag\\={}&\oint_{\partial\mathfrak D}\left\{[\bm \nu\times\nabla\log E(\bm r)]\cdot\left[\nabla\times\frac{\bm n(\bm r)}{E(\bm r)}\right]+\left[\bm \nu\cdot\nabla\log E(\bm r)\right]\left[\nabla\cdot\frac{\bm n(\bm r)}{E(\bm r)}\right]-\left[ \bm \nu\cdot\frac{\bm n(\bm r)}{E(\bm r)} \right]\Delta\log E(\bm r)\right\}\D S,\end{align}where $\bm \nu $ is the outward normal vector with respect to the domain  boundary $\partial \mathfrak D$.

We first look at the integral over $ \mathfrak D$ (which vanishes  in the two-dimensional electrostatics where $\Delta\bm Q=\Delta\bm F=\mathbf0 $).  We can rewrite the integrand as\begin{align}\nabla\log E\cdot\Delta\frac{\bm n}{E}-\frac{\bm n}{E}\cdot\Delta\nabla\log E={}&4\beta^{ij}\partial_i\frac1E\partial_{j}\log E+\frac{8HK}{E}+\frac{2}{E}\bm n\cdot\nabla K\notag\\={}&4\beta^{ij}\partial_i\frac1E\partial_{j}\log E+\frac{8HK}{E}+\frac{2}{E}\bm n\cdot\nabla K-8\bm n\cdot\nabla \frac{K}{E}+\frac{8}{\sqrt{g}}\partial_i\left( \beta^{ij}\sqrt{g}\partial_j\frac1E \right),\end{align}
after employing the relations in  \eqref{eq:Kphi1}, \eqref{eq:dlnE2K_id} and \eqref{eq:Lap_nE}.

We then turn our attention to the boundary contributions. If we pick  the  boundary $ \partial\mathfrak D=\Sigma_{\varphi_1}\cup\Sigma_{\varphi_2}$ as the union of two equipotential surfaces $ \Sigma_{\varphi_1}$ and $ \Sigma_{\varphi_2}$, with the latter surface  enclosing the former, then $\bm \nu $ corresponds to  $ \bm n$ on $ \Sigma_{\varphi_2}$ and  $ -\bm n$ on $ \Sigma_{\varphi_1}$.  Meanwhile, it is  straightforward to compute that \begin{align}
\nabla\times\frac{\bm n}{E}={}&-\nabla\times\frac{\nabla\varphi}{E^{2}}=-\nabla\frac{1}{E^{2}}\times\nabla\varphi=\frac{2\bm n\times\nabla\log E}{E},\\\nabla\cdot\frac{\bm n}{E}={}&-\nabla\cdot\frac{\nabla\varphi}{E^{2}}=-\nabla\frac{1}{E^{2}}\cdot\nabla\varphi=-\frac{2\bm n\cdot\nabla\log E}{E}=-\frac{4H}{E}.
\end{align}

Plugging the results from the last two paragraphs into the vector Green identity, we obtain\begin{align}&2\oint_{\Sigma_{\varphi_2}}\frac{\left\vert\bm n\times\nabla\log E\right\vert^{2}-\left\vert\bm n\cdot\nabla\log E\right\vert^{2}+K}{E}\D S-2\oint_{\Sigma_{\varphi_1}}\frac{\left\vert\bm n\times\nabla\log E\right\vert^{2}-\left\vert\bm n\cdot\nabla\log E\right\vert^{2}+K}{E}\D S\notag\\={}&\int_{\varphi_2}^{\varphi_1}\D\varphi\left\{ \oint_{\Sigma_{\varphi}}\left[4\beta^{ij}\partial_i\frac1E\partial_{j}\log E+\frac{12HK}{E}-6\bm n\cdot\nabla \frac{K}{E}+\frac{8}{\sqrt{g}}\partial_i\left( \beta^{ij}\sqrt{g}\partial_j\frac1E \right)\right]\frac{\D S}{E} \right\}\notag\\={}&\int_{\varphi_2}^{\varphi_1}\D\varphi\left\{ -12\oint_{\Sigma_{\varphi}}\beta^{ij}\partial_i\frac1E\partial_{j}\frac{1}{E}\D S +6\frac{\D}{\D\varphi}\oint_{\Sigma_{\varphi}}\frac{K}{E}\D S \right\}.\label{eq:vecGreen_app}\end{align}Here, in the last step, we have integrated by parts, and used the fact that $ \partial_0\sqrt{g}=2H\sqrt{g}/E$ [see the  second half of \eqref{eq:4HE}]. Now, differentiating both sides of \eqref{eq:vecGreen_app} with respect to $ \varphi_1$, we arrive at \begin{align}8\frac{\D}{\D\varphi}\oint_{\Sigma_{\varphi}}\left( H^2 -K-\frac{|D\log E|^2}{4}\right)\frac{\D S}{E}={}&-2\frac{\D}{\D\varphi}\oint_{\Sigma_{\varphi}}\frac{\left\vert\bm n\times\nabla\log E\right\vert^{2}-\left\vert\bm n\cdot\nabla\log E\right\vert^{2}+4K}{E}\D S\notag\\={}&-12\oint_{\Sigma_{\varphi}}\beta^{ij}\partial_i\frac1E\partial_{j}\frac{1}{E}\D S,\end{align}as claimed in \eqref{eq:W1'-F1'}.
\end{proof}
When we are dealing with the convex boundary surfaces $\partial \Omega $ in Theorem~\ref{thm:quasi-csv}, all the equipotential surfaces  $ \Sigma_\varphi\neq \partial \Omega$ in question are  strictly convex \cite[Theorem 1.1]{JostMaOu2012}, on which  $ (\beta^{ij})=(\crb^{ij}/g)$ is negative definite. Therefore, we have a monotonicity statement\begin{align}
\frac{\D}{\D\varphi}\oint_{\Sigma_{\varphi}}\left( H^2 -K-\frac{|D\log E|^2}{4}\right)\frac{\D S}{E}\geq0,\label{eq:deriv_nonneg}
\end{align}where the inequality is strict unless  $ E(\bm r),\bm r\in\Sigma_\varphi\neq \partial \Omega$ is a constant.

It is worth noting that the last inequality is the first instance where the strict convexity of equipotential surfaces  has played an indispensable r\^ole in our derivations. All our previous theoretical developments are applicable  to both convex and non-convex equipotential surfaces alike. Since  Theorem \ref{thm:quasi-csv} and Corollary \ref{cor:free_bd} both require convex equipotential surfaces,  a diligent reader may rework all our main results in this article using the support function of convex  equipotential surfaces, as in Ma--Zhang \cite{MaZhang2013}.   \subsection{Geometric inequalities and their applications\label{subsec:geo_ineq}}  Our next task is to show that\begin{align}\lim_{\varphi\to0}
\oint_{\Sigma_{\varphi}}\left( H^2 -K-\frac{|D\log E|^2}{4}\right)\frac{\D S}{E}=0\label{eq:3-exD_lim}
\end{align}in 3-exD and \begin{align}
\lim_{\varphi\to+\infty}\label{eq:3-inD_lim}
\oint_{\Sigma_{\varphi}}\left( H^2 -K-\frac{|D\log E|^2}{4}\right)\frac{\D S}{E}=0
\end{align} in 3-inD. Once this is done, we can deduce the two inequalities in Theorem \ref{thm:quasi-csv} from  \eqref{eq:deriv_nonneg}.

As we go to sufficiently large distances $ |\bm r|$ in $ \Omega$, say, away from the circumsphere of $ \mathbb R^3\smallsetminus\Omega$, the spherical harmonic expansion of $U(\bm r)$ converges uniformly and absolutely \cite[\S4.1]{Jackson:EM}:\begin{align}U(\bm r)=\sum_{\ell =0}^\infty\sum_{m=-\ell}^\ell\sqrt{\frac{4\pi}{2\ell+1}}c_{\ell m}\frac{Y_{\ell m}(\theta,\phi)}{|\bm r|^{\ell+1}},\end{align}where the spherical coordinates $  \bm r=|\bm r|(\sin\theta\cos\phi,\sin\theta\sin\phi,\cos\theta)$ are employed, along with  the spherical harmonic function $ Y_{\ell m}(\theta,\phi)$ and  the constants $ c_{\ell m}$ [the multi-pole coefficients associated with the $ (\ell, m)$-modes]. The only significant contributors to our surface integral are the two leading $\ell$-modes: $ \ell=0,1$, as all the higher-order terms  amount to infinitesimal corrections to our surface integral for equipotential surfaces at infinite distances. Without loss of generality, we may evaluate the left-hand side of \eqref{eq:3-exD_lim} by investigating  the dipole field \begin{align}U(\bm r)=\frac{c_{00}}{|\bm r|}+\frac{c_{10}\cos\theta}{|\bm r|^2}, \quad c_{00}>0,c_{10}\neq0,\end{align}which is rotationally symmetric about the $z$-axis. We parametrize the equipotential surface $\Sigma_U$ with\begin{align}\begin{cases}x=\dfrac{c_{00}+\sqrt{(c_{00})^2+4c_{10}U\cos\theta}}{2U}\sin\theta\cos\phi  \\[8pt]
y=\dfrac{c_{00}+\sqrt{(c_{00})^2+4c_{10}U\cos\theta}}{2U}\sin\theta\sin\phi\\[8pt] z=\dfrac{c_{00}+\sqrt{(c_{00})^2+4c_{10}U\cos\theta}}{2U}\cos\theta
\end{cases}(0\leq\theta\leq\pi,0\leq\phi\leq2\pi)\end{align} so that the surface element is given by \begin{align}\D S=\left[\frac{(c_{00})^2 \sin\theta }{U^2}+\frac{c_{10} \sin2 \theta }{U}+O(U^0)\right]\D\theta\D\phi;\end{align}the two principal curvatures (eigenvalues of the Weingarten transform $ \hat W$) read\begin{align}k_\theta=\frac{U}{c_{00}}+\frac{(c_{10})^2 (1-9 \cos 2 \theta) U^3}{4 (c_{00})^5}+O(U^{4}),\quad k_\phi=\frac{U}{c_{00}}-\frac{(c_{10})^2 (5+3 \cos 2 \theta) U^3}{4 (c_{00})^5}+O(U^{4}),\end{align}leading us to \begin{align} H^2-K=\frac{(k_\theta-k_\phi)^{2}}{4}=O(U^{6});\end{align}  the surface distribution of  $ E=E(\theta,\phi)$ satisfies\begin{align}E={}&\frac{U^2}{c_{00}}-\frac{(c_{10})^2  (3 \cos2 \theta+1)U^4 }{4  (c_{00})^5}+O(U^5);\\|\bm n\times\nabla\log E|^{2}={}&\frac{9(c_{10})^4  \sin ^2\theta \cos ^2\theta }{ (c_{00})^{10}}U^6+O(U^7).\end{align}Now it becomes clear that  our integral on $ \Sigma_U$ has order $ O(U^{2})$ for the dipole field.
Hence, the limit formula \eqref{eq:3-exD_lim} is true.

So far, we have established\begin{align}
\oint_{\Sigma_{\varphi}}\left( H^2 -K-\frac{|D\log E|^2}{4}\right)\frac{\D S}{E}=-\frac{3}{2}\int_0^\varphi\left( \oint_{\Sigma_U}\beta^{ij}\partial_{i}\frac{1}{E}\partial_{j}\frac{1}{E}\D S \right)\D U\geq0\label{eq:3-exDc_vol_int}
\end{align} for 3-exDc, where $ (\beta^{ij})$ is negative definite on $\Sigma_{U}$ for all $ U\in(0,\varphi)$.

If the equality holds for a certain given $ \Sigma_{\varphi}$, then we will have $ D\log E(\bm r)=0,\bm r\in\Sigma_{U}$ for all $ U\in(0,\varphi)$, and also\begin{align}
\oint_{\Sigma_{U}}\frac{H^2 -K}{E}\D S=\oint_{\Sigma_{U}}\left( H^2 -K-\frac{|D\log E|^2}{4}\right)\frac{\D S}{E}=0
\end{align} for all $ U\in(0,\varphi)$. This implies that at every point on the strictly convex surface $ \Sigma_{U}$, the two eigenvalues of the Weingarten transform $ \hat W$ are equal, so  $ \Sigma_{U}$ must be a sphere \cite[\S5.2, Theorem 1b]{doCarmo2016}. The condition  $ D\log E(\bm r)=0,\bm r\in\Sigma_{U}$ also means that the spheres $\Sigma_U,U\in(0,\varphi) $ are all concentric. If the center of these spheres is $ \bm r_0\in \mathbb R^3$,  then we will have $ U(\bm r)=\frac{\Phi}{4\pi|\bm r-\bm r_0|}$ whenever $ |\bm r-\bm r_0|>\frac{\Phi}{4\pi \varphi}$. By the unique continuation principle \cite{Aronszajn1957,Protter1960}, we know that  $ U(\bm r)=\frac{\Phi}{4\pi|\bm r-\bm r_0|}$ holds for all  $ \bm r\in\Omega\cup\partial\Omega$. This proves that $\partial\Omega $ is spherical.

After establishing the first half of Theorem \ref{thm:quasi-csv}, we can move on to the 3-exDc case of Corollary \ref{cor:free_bd}.
Suppose that we have a spherical equipotential surface $\Sigma_\varphi $ on which\begin{align}
-\oint_{\Sigma_{\varphi}}\frac{|D\log E|^2}{4E}{\D S}=\oint_{\Sigma_{\varphi}}\left( H^2 -K-\frac{|D\log E|^2}{4}\right)\frac{\D S}{E}\geq0
\end{align}entails $ D\log E(\bm r)=0,\bm r\in\Sigma_{\varphi}$, so equality holds in \eqref{eq:3-exDc_vol_int}. We are then reduced to the situations in the last paragraph, whereupon a spherical   $ \partial \Omega$ becomes inevitable.

 It is much easier to prove the limit formula \eqref{eq:3-inD_lim} for 3-inDc, because \begin{align}
\frac{\D S}{E}=O(|\bm r|^4\sin \theta\D\theta\D\phi),\quad H^2-K=O\left( \frac{1}{|\bm r|^2} \right),\quad |D\log E|=O\left( \frac{1}{|\bm r|} \right),
\end{align}as $ |\bm r|\to0$. This quickly leads us to \begin{align}
\oint_{\Sigma_{\varphi}}\left( H^2 -K-\frac{|D\log E|^2}{4}\right)\frac{\D S}{E}=\frac{3}{2}\int_\varphi^{+\infty}\left( \oint_{\Sigma_G}\beta^{ij}\partial_{i}\frac{1}{E}\partial_{j}\frac{1}{E}\D S \right)\D G\leq0\label{eq:3-inDc_ineq'}
\end{align} for 3-inDc, where $ (\beta^{ij})$ is negative definite on $\Sigma_{G}$ for all $ G\in(\varphi,+\infty)$.

If the equality holds for a certain given $ \Sigma_{\varphi}$, then we will have
 $ D\log E(\bm r)=0,\bm r\in\Sigma_{G}$ for all $ G\in(\varphi,+\infty)$, and also\begin{align}
\oint_{\Sigma_{G}}\frac{H^2 -K}{E}\D S=\oint_{\Sigma_{G}}\left( H^2 -K-\frac{|D\log E|^2}{4}\right)\frac{\D S}{E}=0
\end{align} for all $ G\in(\varphi,+\infty)$. This implies that the strictly convex equipotential surfaces $\Sigma_G,G\in(\varphi,+\infty) $ are  concentric spheres, and that $ G(\bm r)=\frac{1}{4\pi|\bm r|}$ for $ 0<|\bm r|<\frac{1}{4\pi \varphi}$. Again, by unique continuation, we conclude that  $\partial\Omega $ must be a  sphere centered at the origin.

After completing the verification of Theorem \ref{thm:quasi-csv}, we can wrap up our main course with the 3-inDc case of Corollary \ref{cor:free_bd}.
Suppose that we have  $|D\log E(\bm r)|=|\bm n\times\nabla G(\bm r)|=0,\bm r\in\Sigma_\varphi $ so that \begin{align}
\oint_{\Sigma_{\varphi}}\frac{(H^2 -K)\D S}{E}=\oint_{\Sigma_{\varphi}}\left( H^2 -K-\frac{|D\log E|^2}{4}\right)\frac{\D S}{E}\leq0.
\end{align}Since $ H^2-K\geq0$, we are led to $ H^2-K\equiv0$ on $ \Sigma_\varphi$.  Therefore, equality holds in \eqref{eq:3-inDc_ineq'}, and we are reduced to the scenario in the last paragraph, with the same conclusion about the configuration of  $\partial\Omega$.

It is worth noting that Agostiniani and Mazzieri \cite[Appendix A]{AgostinianiMazzieri2015} have furnished us a general framework for asymptotic analysis of exterior and interior Dirichlet problems involving harmonic potentials, applicable to Euclidean spaces of arbitrary dimensions. I thank an anonymous referee for bringing my attention to their work.

\subsection{Related problems in $ \mathbb R^d$ ($d\geq2$)\label{subsec:rel_prob}}In the next theorem, we strengthen the  first statement   of Corollary \ref{cor:free_bd} in $ \mathbb R^d$ ($d>2$).\begin{theorem}[A free boundary problem in $ d$-exD]Let $ d\in\mathbb Z_{>2}$. Suppose that $ U(\bm r),\bm r\in\Omega\subset\mathbb R^d$ solves the Laplace equation in an unbounded domain $ \Omega$, whose boundary $ \partial \Omega$ is a compact (hyper)surface. This function has  asymptotic behavior  $ U(\bm r)\sim \frac{\Phi }{d(d-2)\pi^{d/2}|\bm r|^{d-2}}\int_0^\infty t^{d/2}e^{-t}\D t $ as $ |\bm r |\to+\infty$. If one  equipotential surface in $ \Omega$ is a (hyper)sphere centered at $ \bm r_0\in\mathbb R^d$, then  $ U(\bm r)= \frac{\Phi }{d(d-2)\pi^{d/2}|\bm r-\bm r_0|^{d-2}}\int_0^\infty t^{d/2}e^{-t}\D t$ holds for all $ \bm r\in\Omega$.\end{theorem}\begin{proof}Suppose that we have an equipotential surface $ |\bm r-\bm r_0|=R$. Define
\begin{align}
V(\bm r')\equiv V\left( \frac{R^{2} (\bm r-\bm r_0)}{|\bm r-\bm r_0|^{2}}\right):=|\bm r-\bm r_0|^{d-2}U(\bm r),\quad |\bm r-\bm r_0|\geq R.
\end{align}One can check that this expression extends to a bounded harmonic function $ V(\bm r'),|\bm r'|\leq R$, whose boundary value is a constant. Such a harmonic function must  be a constant function. This proves our claim that all the equipotential surfaces of  $ U(\bm r),\bm r\in\Omega\subset\mathbb R^d$ are (hyper)spherical. \end{proof}

We note that  the second half of Corollary \ref{cor:free_bd}  extends to
  $ d$-inD  (without any convexity requirements) of arbitrary dimensions $d$, as shown by Payne--Schaeffer \cite[Theorem III.2]{PayneSchaefer1989}. Similar results for $ p$-harmonic functions have also been obtained by Alessandrini--Rosset \cite[Theorem 1.1]{AlessandriniRosset2007}, Enciso--Peralta-Salas \cite[Theorem 1]{EncisoPeralta-Salas2009}
and Poggesi \cite[Theorem 1.3]{Poggesi2019}.

Before closing this article, we state and prove the planar analog of Corollary \ref{cor:free_bd} (assuming that  the boundary curves  $\partial\Omega$ are always smooth Jordan curves), using results from \cite{Geom2D}.

\begin{theorem}[Circular solutions to three free boundary value problems]If there is a circular equipotential curve in 2-exD, then the boundary  $ \partial\Omega $ must be a circle. If there is an equipotential curve in 2-exD on which $ E(\bm r)=|\nabla U(\bm r)|$ remains constant, then  $ \partial\Omega $ must be a circle. If there is an equipotential curve in 2-inD on which $ E(\bm r)=|\nabla G(\bm r)|$ remains constant, then  $ \partial\Omega $ must be a circle centered at the origin.\end{theorem}\begin{proof}In \cite[\S2.3]{Geom2D}, we have demonstrated the following inequality (strict unless $ \partial\Omega$ is circular) \begin{align}
\oint_\Sigma \left[\bm n\times\nabla\frac{\kappa}{E}\right]\cdot\left[\bm n\times\nabla\frac{1}{E} \right]\D s\leq0\label{eq:grad_prod}
\end{align}for all equipotential curves $\Sigma $ in 2-exD. Here, the sign convention for curvature has been chosen so that the unit circle has $\kappa=+1 $. If an equipotential curve  $ \Sigma$ is circular, then we will have\begin{align}
\kappa\oint_\Sigma \left[\bm n\times\nabla\frac{1}{E}\right]\cdot\left[\bm n\times\nabla\frac{1}{E} \right]\D s\leq0
\end{align} for a positive constant $\kappa$. This means that  $ \bm n\times\nabla\frac{1}{E(\bm r)}=\mathbf0,\bm r\in \Sigma$ and equality holds in \eqref{eq:grad_prod}. Therefore, the boundary curve  $ \partial\Omega $ is indeed a circle.

For two-dimensional electrostatic problems, we have  $ \kappa=-2H$, $ K\equiv0$ and $ \beta^{ij}\equiv0$. Thus, Proposition \ref{prop:lnE_K} and Corollary \ref{cor:Lap_n_E} reduce to $ \Delta \log E=0$ and $ \Delta\frac{\bm n}{E}=0$, respectively. The vector Green identity in our proof of Corollary \ref{cor:betaij_int} then brings us an integral \begin{align}
\oint_{\Sigma_\varphi}\frac{\kappa^2-|D \log E|^2}{E}\D s
\end{align}  that is independent of $\varphi$. In \cite[\S2.2 and \S3.2]{Geom2D}, we have shown that such a geometric conservation law can be paraphrased as\begin{align}
\oint_{\Sigma_\varphi}\left(\frac{\kappa}{E}-\oint_{\Sigma_\varphi}\frac{\kappa}{E}\D \mu\right)^{2}\D \mu=\oint_{\Sigma_\varphi}\left\vert D\frac{1}{E} \right\vert^{2}\D \mu
\end{align} for a probability measure $ \D\mu=E\D s/\Phi $.
In both 2-exD and 2-inD, plugging a constant field intensity $ D\frac{1}{E(\bm r)}=\mathbf 0,\bm r\in\Sigma_\varphi$ into the right-hand side of the equation above, we may read off from the left-hand side  that $ D\frac{\kappa}{E}=0$ on the respective equipotential curve. This implies that there is an equality\begin{align}
\oint_{\Sigma_{\varphi}} \left[\bm n\times\nabla\frac{\kappa}{E}\right]\cdot\left[\bm n\times\nabla\frac{1}{E} \right]\D s=0.
\end{align} According to our analysis in \cite[\S2.3 and \S3]{Geom2D}, this can only happen if  $ \partial \Omega$ is a circle in 2-exD, or   $ \partial \Omega$ is a circle centered at the origin  in 2-inD.    \end{proof}

\subsection*{Acknowledgments}Part of this work was assembled from  my research notes in 2006 (on curvature effects in nanophotonics) and 2011 (on entropy in curved spaces). I am grateful to Prof.\ Xiaowei Zhuang (Harvard) and Prof.\ Weinan E (Princeton) for their thought-provoking questions in 2006 and 2011 that inspired these research notes.
Many thanks are also due to two referees whose suggestions helped improving the presentation of the current work.

\end{document}